\documentclass{article}

\usepackage{graphicx,xcolor,textpos}
\usepackage{helvet}
\usepackage{amssymb,latexsym,amsmath,epsfig,amsthm,hyperref,tipa,comment}
\newtheorem{lowerdensitytheory}{Theorem}
\newtheorem{minimumes}[lowerdensitytheory]{Theorem}
\newtheorem{lowerdensitypractice}{Corollary}
\newtheorem{densityconjecture}{Conjecture}
\newtheorem{miseight}[lowerdensitytheory]{Theorem}

\newtheorem{logjes}{Lemma}
\newtheorem{pdisjoint}{Lemma}
\newtheorem{fourtimeseven}[logjes]{Lemma}
\newtheorem{fourtimesodd}[logjes]{Lemma}
\newtheorem{eveneight}[logjes]{Lemma}
\newtheorem{oddeight}[logjes]{Lemma}
\newtheorem{mistwoorfour}[logjes]{Lemma}
\newtheorem{setsoffour}[lowerdensitypractice]{Corollary}
\newtheorem{fngrows}[logjes]{Lemma}
\newtheorem{logbylinear}[logjes]{Lemma}
\newtheorem{basiclog}[logjes]{Lemma}
\newtheorem{taylorlog}[logjes]{Lemma}
\newtheorem{taylore}[logjes]{Lemma}
\newtheorem{nastylog}[logjes]{Lemma}
\newtheorem{anothernastylog}[logjes]{Lemma}
\newtheorem{upperbounde}[logjes]{Lemma}
\newtheorem{basicloglog}[logjes]{Lemma}
\newtheorem{notherone}[logjes]{Lemma}
\newtheorem{growthweak}[lowerdensitytheory]{Theorem}

\begin{document}

\vspace*{-2cm}

\Large
 \begin{center}
On the congruence properties and growth rate of a recursively defined sequence \\ 

\hspace{10pt}

\large
Wouter van Doorn \\

\hspace{10pt}

\end{center}

\hspace{10pt}

\normalsize

\vspace{-10pt}

\centerline{\bf Abstract}

Let $a_1 = 1$ and, for $n > 1$, $a_n = a_{n-1} + a_{\left \lfloor \frac{n}{2} \right \rfloor}$. In this paper we will look at congruence properties and the growth rate of this sequence. First we will show that if \mbox{$x \in \{1, 2, 3, 5, 6, 7 \}$}, then the natural density of $n$ such that $a_n \equiv x \pmod{8}$ exists and equals $\frac{1}{6}$. Next we will prove that if $m \le 15$ is not divisible by $4$, then the lower density of $n$ such that $a_n$ is divisible by $m$, is strictly positive. To put these results in a broader context, we will then posit a general conjecture about the density of $n$ such that $a_n \equiv x \pmod{m}$ for any given $x$ and any $m$ not divisible by $32$. Finally, we will show that there exists a function $f$ such that $n^{f(n)} < a_n < n^{f(n) + \epsilon}$ for all $\epsilon > 0$ and all large enough $n$. 

\section{Introduction}
We are interested in the following sequence: $a_1 = 1$ and, for $n > 1$, $a_n = a_{n-1} + a_{\left \lfloor \frac{n}{2} \right \rfloor}$. In this sequence, the term $a_n$ counts the number of partitions of $2n$ into so-called strongly decreasing parts; see \mbox{\cite{bos}} for more information. In \mbox{\cite{gkm}} the congruence properties of $a_n$ were studied, where in particular, Theorem $3.1$ in \mbox{\cite{gkm}} shows that $a_n$ is never divisible by $4$. On the other hand, it can be shown that for $m \in \{2, 3, 5, 7\}$ there are infinitely many $n$ such that $a_n \equiv 0 \pmod{m}$. The result that $a_n$ is divisible by $7$ for infinitely many $n$ is mentioned in \mbox{\cite{oly}}, a book on problems from Mathematical Olympiads. \\

Generalizing this, for given positive integers $x$ and $m$, one can ask how many $n$ there are such that $a_n \equiv x \pmod{m}$. To this end, let $S_{x,m}(n)$ be the number of $k \in \{1, 2, \ldots, n\}$ such that $a_k \equiv x \pmod{m}$. We then define the limit $d_{x,m} = \displaystyle \lim_{n \rightarrow \infty} \dfrac{S_{x,m}(n)}{n}$ and the limit inferior $\underline{d_{x,m}} = \displaystyle \liminf_{n \rightarrow \infty} \dfrac{S_{x,m}(n)}{n}$. The former may of course not exist, but the latter always will. As mentioned, in \mbox{\cite{gkm}} it is shown that $d_{0,4} = 0$. In Section \ref{modeight} we will extend their result and show that $d_{x,8} = \frac{1}{6}$ whenever $4$ does not divide $x$. \\

For a given modulus $m$, let $x \not\equiv 0 \pmod{4}$ be divisible by the largest odd divisor of $m$. In Section \ref{modthree} we will prove a theorem that implies that checking finitely many cases can be sufficient in order to deduce $\underline{d_{x,m}} > 0$. And with a computer we did indeed check sufficiently many cases for all $m \le 15$ to show non-trivial lower bounds for these lower densities. In Section \ref{modalles} we will then propose a general conjecture for the existence and value of $d_{x,m}$ for all $m \not \equiv 0 \pmod{32}$. Finally, in Section \ref{growthrate} we will prove reasonably tight bounds on $a_n$ as a function of $n$, where the ratio of upper and lower bounds grows slower than $n^{\epsilon}$ for any $\epsilon > 0$. 

\newpage
\section{Congruence classes modulo 8} \label{modeight}
The goal of this section is to show the following result:

\begin{miseight} \label{miseight}
For all $n \in \mathbb{N}$ and for $x \in \{1, 2, 3, 5, 6, 7\}$ we have $S_{x,8}(n) > \frac{1}{6}n - 2\log(n) - 11$. In particular, $d_{x,8} = \frac{1}{6}$ if $x$ is not divisible by $4$.
\end{miseight}

To prove this theorem, we will first determine exactly when $a_n$ is even. To this end, let $2^r$ be the largest power of $2$ dividing $n$. Since $a_1 = 1$ and, for $n \ge 1$ we have $a_{2n+1} = a_{2n} + a_{n} = a_{2n-1} + 2a_n \equiv a_{2n-1} \pmod{2}$, it is by induction easily seen that, if $n$ is odd (i.e. if $r$ equals $0$), then $a_n$ is odd. This will be the base step in the induction proof of the following:

\begin{mistwoorfour} \label{even}
The parity of $a_n$ only depends on the parity of $r$. More precisely, $a_n \equiv r+1 \pmod{2}$.
\end{mistwoorfour}

\begin{proof}
Write $n =  2^r \cdot s$ for some positive integers $r$ and $s$, where $s$ is odd. We then get:
\begin{align*}
a_{2^r \cdot s} &= a_{2^r \cdot s - 1} + a_{2^{r-1} \cdot s} \\
&\equiv 1 + (r-1) + 1 \pmod{2} \\
&\equiv r+1 \hspace{42pt} \pmod{2}
\end{align*}
\end{proof}

Note that Lemma \ref{even} implies that $d_{0,2} = \frac{1}{3}$. Indeed, we see that $a_n$ is even precisely when $n \equiv 2 \pmod{4}$, or $8 \pmod {16}$, or $32 \pmod {64}$, etc. In other words, the density of $n$ such that $a_n$ is even, equals $\frac{1}{4} + \frac{1}{16} + \frac{1}{64} + \ldots = \frac{1}{3}$.\\

Lemma \ref{even} will be able to help us prove a few lemmas. Two lemmas that state what happens to $a_n \pmod{8}$ if we multiply $n$ by four, and two lemmas that provide us with base cases.

\begin{fourtimesodd} \label{fourtimesodd}
If $n \in \mathbb{N}$ is odd, $a_{4n} \equiv a_{n} + 4 \pmod{8}$.
\end{fourtimesodd}

\begin{fourtimeseven} \label{fourtimeseven}
If $n \in \mathbb{N}$ is even, $a_{4n} \equiv a_n \pmod{8}$.
\end{fourtimeseven}

\begin{oddeight} \label{oddeight}
For all non-negative integers $n$ the set $\{a_{8n+1}, a_{8n+3}, a_{8n+5}, a_{8n+7} \}$ contains all odd residue classes modulo $8$ exactly once. 
\end{oddeight}

\begin{eveneight} \label{eveneight}
For all non-negative integers $n$ we have the string of equalities $a_{16n+2} \equiv a_{16n+6} \equiv -a_{16n+10} \equiv -a_{16n+14} \equiv \pm 2 \pmod{8}$. In particular, the set $\{a_{16n+2}, a_{16n+6}, a_{16n+10}, a_{16n+14}\}$ contains both even residue classes modulo $8$ that are not divisible by $4$ exactly twice. 
\end{eveneight}

\begin{proof}[Proof of Lemmas \ref{fourtimesodd} and \ref{fourtimeseven}]
Let $k = 4$ when $n$ is odd and $k = 0$ when $n$ is even. Then $a_{4n} \equiv a_{n}+k \pmod{8}$ can be checked for $n = 1$ and $n = 2$. Now we will use induction again, so assume $a_{4n} \equiv a_{n} + k \pmod{8}$ for some $n \in \mathbb{N}$. We will then prove that $a_{4(n+2)} \equiv a_{n+2} + k \pmod{8}$ as well. First of all, let us simply use the definition of our sequence a couple of times.

\begin{align*}
a_{4n+8} &= a_{4n+7} + a_{2n+4} \\
&= a_{4n+6} + a_{2n+4} + a_{2n+3} \\
&= a_{4n+5} + a_{2n+4} + 2a_{2n+3} \\
&= a_{4n+4} + a_{2n+4} + 2a_{2n+3} + a_{2n+2} \\
&= a_{4n+3} + a_{2n+4} + 2a_{2n+3} + 2a_{2n+2} \\
&= a_{4n+2} + a_{2n+4} + 2a_{2n+3} + 2a_{2n+2} + a_{2n+1} \\
&= a_{4n+1} + a_{2n+4} + 2a_{2n+3} + 2a_{2n+2} + 2a_{2n+1} \\
&= a_{4n} + a_{2n+4} + 2a_{2n+3} + 2a_{2n+2} + 2a_{2n+1} + a_{2n} \\
&= a_{4n} + a_{n+2} + 3a_{2n+3} + 2a_{2n+2} + 2a_{2n+1} + a_{2n}
\end{align*}

Now that we have reached this equality, we can use our induction hypothesis $a_{4n} \equiv a_{n} + k \pmod{8}$. We will then apply $a_{n} + a_{2n} = a_{2n+1}$ and use the fact that $a_{2n+2}$ is the average of $a_{2n+3}$ and $a_{2n+1}$ to finish the proof.

\begin{align*}
a_{4n+8} &= a_{4n} + a_{n+2} + 3a_{2n+3} + 2a_{2n+2} + 2a_{2n+1} + a_{2n} \\
&\equiv a_n + k + a_{n+2} + 3a_{2n+3} + 2a_{2n+2} + 2a_{2n+1} + a_{2n} \pmod{8} \\
&= a_{n+2} + k + 3a_{2n+3} + 2a_{2n+2} + 3a_{2n+1} \\
&= a_{n+2} + k + 8a_{2n+2} \\
&\equiv a_{n+2} + k \pmod{8} \qedhere
\end{align*}
\end{proof}

\begin{proof}[Proof of Lemma \ref{oddeight}]
First note that both $a_{4n+1}$ and $a_{8n+1}$ are odd by Lemma \ref{even}. It then follows that $a_{4n+1} \equiv a_{8n+1} \pmod{4}$, as we would otherwise get $a_{8n+2} = a_{8n+1} + a_{4n+1} \equiv 0 \pmod{4}$, which is impossible by Theorem $3.1$ in \mbox{\cite{gkm}}. Now we will use this fact and apply the equalities $2a_{4n+2} \equiv 4 \pmod{8}$ and $2a_{4n+1} + 2a_{4n+3} = 4a_{4n+2}$ to show that $a_{8n+1}, a_{8n+3}, a_{8n+5},$ and $a_{8n+7}$ are all distinct when reduced modulo $8$.

\begin{align*}
a_{8n+1} &\equiv a_{8n+1} \pmod{8} \\
a_{8n+3} &= a_{8n+1} + 2a_{4n+1} \equiv 3a_{8n+1} \pmod{8}  \\
a_{8n+5} &= a_{8n+3} + 2a_{4n+2} \equiv 3a_{8n+1} + 4 \equiv 7a_{8n+1} \pmod{8} \\
a_{8n+7} &= a_{8n+1} + 2a_{4n+1} + 2a_{4n+2} + 2a_{4n+3} = a_{8n+1} + 6a_{4n+2}  \equiv a_{8n+1} + 4 \equiv 5a_{8n+1} \pmod{8} \\
\end{align*}

\end{proof}

\newpage
\begin{proof}[Proof of Lemma \ref{eveneight}]
By Lemma \ref{even} and the fact that no element of the sequence is divisible by $4$, we know that $a_{16n+4k+2} \equiv \pm 2 \pmod{8}$ for all $k \in \{0,1,2,3\}$. It therefore suffices to show that $a_{16n+4k+2} \equiv a_{16n+4k-2} \pmod{8}$ for $k \in \{1,3\}$ and $a_{16n+4k+2} \equiv a_{16n+4k-2} + 4 \pmod{8}$ for $k = 2$.

\begin{align*}
a_{16n+4k+2} &= 
 a_{16n+4k-2} + a_{8n+2k+1} + 2a_{8n+2k} + a_{8n+2k-1} \\ 
&= a_{16n+4k-2} + 4a_{8n+2k} \\ 
\end{align*}

Since $a_{8n+2k}$ is even for $k \in \{1,3\}$ and odd for $k = 2$, this finishes the proof.
\end{proof}

By combining the four lemmas we just proved, we obtain the following corollary:

\begin{setsoffour}
For all non-negative integers $k$ and $m$ the following set contains all odd residues classes modulo $8$ exactly once: 

\begin{center}
$P_{k,m} := \{a_{4^m(8k+1)}, a_{4^m(8k+3)}, a_{4^m(8k+5)}, a_{4^m(8k+7)} \}$ 
\end{center}

On the other hand, for all non-negative integers $k$ and $m$ the following set contains two integers congruent to $2 \pmod{8}$ and two integers congruent to $6 \pmod{8}$: 

\begin{center}
$Q_{k,m} := \{a_{4^m(16k + 2)}, a_{4^m(16k + 6)}, a_{4^m(16k + 10)}, a_{4^m(16k + 14)} \}$
\end{center}
\end{setsoffour}

The proof of Theorem \ref{miseight} is now no longer very difficult.

\begin{proof}[Proof of Theorem \ref{miseight}]
For $n < 64$ the theorem is true as the right-hand side is negative, so let us assume $n \ge 64 = 4^3$. This means there exists a non-negative integer $l$ such that $4^{l+3} \le n < 4^{l+4}$. Define the sets $P_{k,m}$ as above for $0 \le m \le l$ and $0 \le k \le \left \lfloor \frac{n}{2 \cdot 4^{m+1}} \right \rfloor - 1$, and define the sets $Q_{k,m}$ as above for $0 \le m \le l$ and $0 \le k \le \left \lfloor \frac{n}{4^{m+2}} \right \rfloor - 1$. Note that all elements of $P_{k,m}$ and $Q_{k,m}$ are smaller than $n$ and that all of these sets are pairwise disjoint. Now we will find lower bounds on $S_{x,8}(n)$ for $x \in \{1, 3, 5, 7\}$ and on $S_{y,8}(n)$ for $y \in \{2, 6\}$, to finish off the proof.

\begin{align*}
S_{x,8}(n) &\ge \#P_{k,m} \\
&=\sum_{m = 0}^l \sum_{k = 0}^{\left \lfloor \frac{n}{2 \cdot 4^{m+1}} \right \rfloor - 1} 1 \\
&\ge \sum_{m = 0}^l \left(\frac{n}{2 \cdot 4^{m+1}} - 1\right) \\
&= \frac{n}{6} - (l+1) - \frac{n}{6 \cdot 4^{l+1}}  \\
&> \frac{n}{6} - \log(n) - 11
\end{align*}

\begin{align*}
S_{y,8}(n) &\ge 2\#Q_{k,m} \\
&= 2 \sum_{m = 0}^l \sum_{k = 0}^{ \left \lfloor \frac{n}{4^{m+2}} \right \rfloor - 1} 1 \\
&\ge \sum_{m = 0}^l \left(\frac{n}{2 \cdot 4^{m+1}} - 2\right) \\
&= \frac{n}{6} - 2(l+1) - \frac{n}{6 \cdot 4^{l+1}} \\
&> \frac{n}{6} - 2\log(n) - 11 \qedhere
\end{align*}

\end{proof}

\section{Congruence classes for other moduli} \label{modthree}
For a positive integer $n$, let $k$ be an integer with $\frac{n}{8} \le k \le \frac{n}{4} - 1$, and consider the set $P_k := \{a_{2k+1}, a_{4k+1}, a_{4k+2}, a_{4k+3}\} \subset \{a_3, a_2, \ldots, a_n\}$. Now, if $a_{2k+1}$ is not divisible by $3$, then one can check that $a_{4k+1}, a_{4k+2}, a_{4k+3}$ are all distinct modulo $3$. In other words, $P_k$ contains at least one integer divisible by $3$. Since it can furthermore be checked that the sets $P_k$ are pairwise disjoint for all $k$ with $\frac{n}{8} \le k \le \frac{n}{4}-1$, it follows that $S_{0,3}(n) \ge \left \lfloor \frac{n}{4} - 1 \right \rfloor - \left \lceil \frac{n}{8} \right \rceil > \frac{n}{8} - 3$, showing a positive lower density of $n$ such that $3$ divides $a_n$. Let us now streamline and generalize this argument. \\

Let $j$ be any positive integer and let $l$ be $\left\lfloor \frac{\log(n)}{\log(2^j)} \right \rfloor$. Then, for all positive integers $i$, $k$ and $w$ with $i \le j$, $\frac{n}{2^{jw+1}} \le k \le \frac{n}{2^{jw}} - 1$ and $w \le l$, define the following sets containing $2^i - 1$ consecutive elements of our sequence starting at $a_{2^ik+1}$: 

\begin{center}
$P_{i,k,w} = \{a_{2^ik + 1}, a_{2^ik + 2}, \ldots, a_{2^ik + 2^{i} - 1}\}$
\end{center}

\begin{pdisjoint} \label{disjoint}
For all $i, k, w$ we have $P_{i, k, w} \subset \{a_3, a_4, \ldots, a_n\}$, and the sets $P_{i,k,w}$ are all pairwise disjoint.
\end{pdisjoint}

\begin{proof}
Since $i$ and $k$ are positive integers, we get that $2^ik + 1$, which is the smallest index in $P_{i, k, w}$, is at least equal to $2^1 \cdot 1 + 1 = 3$. On the other hand, the largest index in $P_{i, k, w}$ is $2^ik + 2^{i} - 1 \le 2^jk + 2^j - 1$, which is at most $2^j(\frac{n}{2^{jw}} - 1) + 2^j - 1 < \frac{n}{2^{j(w-1)}} \le n$. \\

To see that these sets are parwise disjoint, there are three possibilities to consider when comparing different sets $P_{i,k,w}$ and $P_{i',k',w'}$. In all cases we will prove that the largest index in one of them is smaller than the smallest index in the other one. Since $a_n$ is an increasing function of $n$, we conclude that all elements in one of the sets must be larger than all of the elements in the other set, and they are therefore disjoint. The three different cases we will consider are: either $w \neq w'$, or $w = w'$ and $k \neq k'$, or $w = w'$, $k = k'$ and $i \neq i'$.  \\

Case $1$: $w \neq w'$. Assume without loss of generality $w < w'$. The smallest index in $P_{i,k,w}$ is then equal to $2^ik + 1$, which is at least $2k + 1 \ge \frac{n}{2^{jw}} + 1$. On the other hand, the largest index in $P_{i', k', w'}$ is equal to $2^{i '}k' + 2^{i'} - 1$, which is at most $2^jk' + 2^{j} - 1 \le 2^j\left(\frac{n}{2^{jw'}} - 1\right) + 2^j - 1 = \frac{n}{2^{j(w'- 1)}} - 2^j + 2^j - 1 \le \frac{n}{2^{jw}} - 1$. \\

Case $2$: $w = w, k \neq k'$. Assume without loss of generality $k > k'$. If $i \ge i'$, the smallest index in $P_{i,k,w}$ is equal to $2^ik + 1$, whereas the largest index in $P_{i', k', w'}$ is equal to $2^{i'}k' + 2^{i'} - 1$, which is at most $2^i(k-1) + 2^i - 1 = 2^ik - 1$. On the other hand, if $i < i'$, note that $k \le \frac{n}{2^{jw}} - 1 = \frac{2n}{2^{jw+1}} - 1 \le 2k' - 1$. Therefore in this case, the largest index in $P_{i,k,w}$ is equal to $2^ik + 2^{i} - 1 \le 2^{i}(2k' - 1) + 2^i - 1 < 2^{i+1}k' \le 2^{i'}k'$, which is smaller than the smallest index in $P_{i',k',w'}$.  \\

Case $3$: $w = w, k = k', i \neq i'$. Assume without loss of generality $i > i'$. The smallest index in $P_{i,k,w}$ is then equal to $2^ik + 1$, whereas the largest index in $P_{i', k', w'}$ is equal to $2^{i'}k' + 2^{i'} - 1$, which is at most $2^{i-1}k' + 2^{i-1} - 1 < 2^{i-1}(k'+1) \le 2^{i-1}(2k') = 2^ik' = 2^ik$.

\end{proof}

We now define $P_{k,w}$ as the union of $P_{i,k,w}$ over all $i$ with $1 \le i \le j$. By Lemma \ref{disjoint} all the sets $P_{k,w}$ have the same number of elements and are pairwise disjoint as well. Further observe that if the $j$ elements $a_{2k+1}, a_{4k+1}, \ldots, a_{2^jk+1}$ (which we will call the base elements) are known, then all elements of $P_{k,w}$ are known, as all elements of $P_{k,w}$ can be written as a sum of base elements, by the definition of our sequence $a_n$. Moreover, the way to write an element of $P_{k,w}$ in terms of its base elements is independent of the exact values of $k$ and $w$. Since none of the base elements are ever even by Lemma \ref{even}, let $m_{*}$ be the largest odd divisor of $m$. If we now fix $j$, then for a given modulus $m$ there are $m_{*}^{j}$ possible different values for the base elements modulo $m$ within a set $P_{k,w}$. Every one of these $m_{*}^j$ options can lead to a different set of values of the elements in $P_{k,w}$ modulo $m$. Now we remark that if, for example, for every one of these $m_{*}^j$ options we obtain at least one element that is congruent to $x \pmod{m}$, then it follows that all sets $P_{k,w}$ must contain an element that is congruent to $x \pmod{m}$! And we claim that this leads to a lower bound on $d_{x,m}$.

\begin{lowerdensitytheory} \label{densitytheory}
Let $e_{x,m,j}$ be the minimum number of elements in $P_{k,w}$ that are congruent to $x \pmod{m}$, where the minimum is taken over all $m_{*}^j$ possible different values of the base elements modulo $m$. Then we get the following lower bound:

\begin{equation*} 
S_{x,m}(n) > \left(\frac{e_{x,m,j}}{2^{j+1} - 2}\right)n - \left(\frac{3e_{x,m,j}}{\log(2^j)}\right)\log(n) - e_{x,m,j}
\end{equation*}

In particular, $\underline{d_{x,m}} \ge \frac{e_{x,m,j}}{2^{j+1} - 2}$. 
\end{lowerdensitytheory}

\begin{proof}
Since every set $P_{k,w}$ contains at least $e_{x,m,j}$ elements that are congruent to $x \pmod{m}$ by assumption, all we need to do is find a lower bound on the number of different sets $P_{k,w}$ there are. This will be done similarly to how we proved Theorem \ref{miseight} earlier;

\begin{align*}
S_{x,m}(n) &\ge e_{x,m,j}\#P_{k,w} \\
&= e_{x,m,j} \sum_{w = 1}^l \sum_{k = \left\lceil \frac{n}{2^{jw+1}} \right \rceil}^{ \left \lfloor \frac{n}{2^{jw}} \right \rfloor - 1} 1 \\
&> e_{x,m,j} \sum_{w = 1}^l \left(\frac{n}{2^{jw+1}} - 3\right) \\
&= \frac{e_{x,m,j}n}{2^{j+1} - 2} - 3e_{x,m,j}l - \frac{e_{x,m,j}n}{2^{jl} \cdot (2^{j+1} - 2)} \\
&> \frac{e_{x,m,j}n}{2^{j+1} - 2} - \frac{3e_{x,m,j}\log(n)}{\log(2^j)} - e_{x,m,j}
\qedhere
\end{align*}
\end{proof}

To apply Theorem \ref{densitytheory} we choose, for a given residue class $x \pmod{m}$, a suitable value of $j$ for which we can, possibly with the help of a computer, brute-force over all $m_*^j$ possible different values of the base elements modulo $m$ and hope that it leads to a large value of $\frac{e_{x,m,j}}{2^{j+1} - 2}$. We have indeed done this for all $m \le 15$, and the following values were all found with a computer: \\

$e_{0,3,13} = 4708$, $e_{0,5,10} = 235$, $e_{0,6,13} = 1130$, $e_{3,6,13} = 2905$, $e_{0,7,9} = 91$, $e_{0,9,8} = 26$, $e_{0,10,11} = 115$, $e_{5,10,11} = 353$, $e_{0,11,8} = 20$, $e_{3,12,13} = 1284$, $e_{6,12,13} = 1130$, $e_{9,12,13} = 1284$, $e_{0,13,7} = 3$, $e_{0,14,9} = 11$, $e_{7,14,9} = 44$, $e_{0,15,7} = 1$. \\

From these values, one can apply Theorem \ref{densitytheory} to obtain positive lower densities;

\begin{minimumes} \label{minimumes}
For all $m \in \{3, 5, 6, 7, 9, 10, 11, 12, 13, 14, 15\}$ and all $x \not \equiv 0 \pmod{4}$ divisible by the largest odd divisor of $m$, the lower density of $n$ such that $a_n \equiv x \pmod{m}$ is strictly positive. More precisely:

\begin{align*}
\underline{d_{0,3}} &\ge \frac{4708}{2^{14} - 2} > 0.2873 & \underline{d_{0,5}} &\ge \frac{235}{2^{11} - 2} >  0.1148 \\
\underline{d_{0,6}} &\ge \frac{1130}{2^{14} - 2} >  0.0689 &\underline{d_{3,6}} &\ge \frac{2905}{2^{14} - 2} >  0.1773 \\
\underline{d_{0,7}} &\ge \frac{91}{2^{10} - 2} >  0.0890 & \underline{d_{0,9}} &\ge \frac{26}{2^{9} - 2} >  0.0509 \\
\underline{d_{0,10}} &\ge \frac{115}{2^{12} - 2} > 0.0280 & \underline{d_{5,10}} &\ge \frac{353}{2^{12} - 2} > 0.0862 \\
\underline{d_{0,11}} &\ge \frac{20}{2^{9} - 2} > 0.0392 & \underline{d_{3,12}} &\ge \frac{1284}{2^{14} - 2} > 0.0783 \\
\underline{d_{6,12}} &\ge \frac{1130}{2^{14} - 2} > 0.0689 & \underline{d_{9,12}} &\ge \frac{1284}{2^{14} - 2} > 0.0783 \\
\underline{d_{0,13}} &\ge \frac{3}{2^{8} - 2} > 0.0118 & \underline{d_{0,14}} &\ge \frac{11}{2^{10} - 2} > 0.0107 \\
\underline{d_{7,14}} &\ge \frac{44}{2^{10} - 2} > 0.0430 & \underline{d_{0,15}} &\ge \frac{1}{2^{8} - 2} > 0.0039
\end{align*}
\end{minimumes}

As a remark, recall that $a_{4k+1} \equiv a_{8k+1} \pmod{4}$ was shown in the proof of Lemma \ref{oddeight} and, more generally, it can be seen that all base elements must be congruent to each other modulo $4$. This can be used to speed up the code when trying to find $e_{x,m,j}$ for any $m$ divisible by $4$. As for finding $e_{x,m,j}$ for more values of $x$ and $m$, note that it is a priori possible that all base elements are congruent to $m^{*} \pmod{m}$, in which case all elements of $P_{k, w}$ are divisible by $m^*$. We therefore see that $e_{x,m,j} = 0$ unless $m^{*}$ divides $x$, and it follows that Theorem \ref{densitytheory} is not strong enough to show a positive lower density for any $x$ not divisible by $m^{*}$.

\newpage
\section{A general conjecture} \label{modalles}
It seems plausible that Theorem \ref{miseight} can be extended to include all residue classes modulo $16$. Furthermore, when $m$ is odd there is a priori no reason to expect $a_n$ to favour certain residue classes modulo $m$ over others. However, by some computer calculations, it looks as though unexpected things might pop up for $m = 32$. For example, it is not even clear that $d_{x,32}$ exists when $x$ is odd. That being said, by checking $e_{x,32,24}$ with a computer and applying Theorem \ref{densitytheory}, one can check $\underline{d_{x,32}} > 0.38$ for all odd $x$. And this is not far off the value of $\frac{1}{24} \approx 0.417$ which would be the naive guess. It is therefore not very clear at the moment what $d_{x,32}$ should be and if it even exists, although it might not be too difficult to resolve this issue one way or another. In any case, the following conjecture for $m$ not a multiple of $32$ does suggest itself:

\begin{densityconjecture} \label{conj}
If $x$ and $m$ are positive integers and $32$ does not divide $m$, then the limit $d_{x,m}$ exists and equals 

$$d_{x,m} = \begin{cases}

\frac{1}{m} &\mbox{if } m \equiv 1 \pmod{2} \\
\frac{2}{3m} &\mbox{if } x \equiv 0 \pmod{2} \mbox{ and } m \equiv 2 \pmod{4} \\
\frac{4}{3m} &\mbox{if } x \equiv 1 \pmod{2} \mbox{ and } m \equiv 2 \pmod{4} \\
\frac{4}{3m} &\mbox{if }  x \not \equiv m \equiv 0 \pmod{4} \\ 
0 &\mbox{if } x \equiv m \equiv 0 \pmod{4} 

\end{cases} $$

\end{densityconjecture}

Morally speaking, this conjecture states: after taking into consideration that only a third of $a_n$ are even and none are divisible by $4$, $a_n \pmod{m}$ is essentially random. Note that Theorem \ref{miseight} proves this conjecture for $m$ equal to (a divisor of) $8$ and that this conjecture is an ambitious version of Conjecture 2.2 from \mbox{\cite{gkm}}. \\

To give the appearance of cautiousness, we should add that the evidence for Conjecture \ref{conj} could be considered a bit tenuous. It is true for $m = 8$, it would be nice if it were true in general, we have found no reason for it not to be true in general, and computer calculations up to $m = 100$ and $n = 10^8$ do seem to support it.

\newpage

\section{Growth rate} \label{growthrate}
Finally, let us move away from congruence properties and switch our focus to determining how fast $a_n$ grows as a function of $n$.

\begin{growthweak} \label{growthweak}
Define the constants $c_1 = \frac{1}{2 \log(2)} \approx 0.721$ and $c_2 = \frac{1}{2} + \frac{1 + \log(\log(2))}{\log(2)} \approx 1.414$, and let $f(n) = c_1 \log(n) - 2c_1\log(\log(n)) + c_2$. Then for every $n \ge 141$ we have $a_n > n^{f(n)}$. On the other hand, for every $\epsilon > 0$ there exists an $C_{\epsilon}$ such that for all $n \ge 2$ we have $a_n < C_{\epsilon}n^{f(n) + \epsilon}$.
\end{growthweak}

Before we can start, we need some quick lemmas.

\begin{fngrows} \label{fngrows}
For all $x \in \mathbb{R}$ with $x > 8$, $f(x)$ is an increasing function of $x$.
\end{fngrows}

\begin{logbylinear} \label{logbylinear}
For all $x \in \mathbb{R}$ with $x > 200$ we have $\frac{2c_1\log(x)}{x} < 0.05$.
\end{logbylinear}

\begin{basiclog} \label{basiclog}
For all $a,x \in \mathbb{R}$ with $a > 0$ and $x > 1$ we have $\log(x) + \frac{a}{x+a} < \log(x+a) < \log(x) + \frac{a}{x}$.
\end{basiclog}

\begin{taylorlog} \label{taylorlog}
For all $x \in \mathbb{R}$ with $x > 1$ we have $\log\left(1 + \frac{1}{x}\right) > \frac{1}{x} - \frac{1}{2x^2}$.
\end{taylorlog}

\begin{taylore} \label{taylore}
For all $x \in \mathbb{R}$ with $x \ge \frac{1}{2}$ we have $e^{\frac{-1}{2x}} > 1 - \frac{1}{2x}$.
\end{taylore}

\begin{nastylog} \label{nastylog}
For all $n \in \mathbb{N}$ with $n \ge 4$ we have $\left(1 + \frac{\log(2)}{\log(\frac{n}{2})}\right)^{\frac{\log(\frac{n}{2})}{\log(2)}} > e - \frac{1}{\log(n)}$.
\end{nastylog}

\begin{upperbounde} \label{upperbounde}
For all $n \in \mathbb{N}$ we have $e < \left(1 + \frac{1}{n}\right)^{n+1}$.
\end{upperbounde}

\begin{anothernastylog} \label{anothernastylog}
For all $n \in \mathbb{N}$ with $n \ge 2$ we have $2c_1\log(\log(\frac{n+1}{2})) > 2c_1\log(\log(n)) - \frac{1}{\log\left(\frac{n+1}{2}\right)}$.
\end{anothernastylog}

\begin{basicloglog} \label{basicloglog}
For all $n \in \mathbb{N}$ with $n \ge 2$ we have $\log(\log(n+1)) - \log(\log(n)) < \frac{1}{n\log(n)}$.
\end{basicloglog}

\begin{notherone} \label{notherone}
For all $n \in \mathbb{N}$ with $n \ge 4$ we have $\frac{2c_1(n+2)\log(n+1)}{n\log(n)} < 3$.
\end{notherone}

All of these lemmas are arguably basic and we only include their proofs for completeness' sake.

\begin{proof}[Proof of Lemma \ref{fngrows}]
Since $f(x)$ is an increasing function if, and only if, $g(x) = 2^{f(x) - c_2}$ is an increasing function, it is sufficient to prove that the latter increases for $x > 8$. 

\begin{align*}
g(x) &= 2^{f(x) - c_2} \\
&= e^{\log(2)(c_1 \log(x) - 2c_1\log(\log(x)))} \\
&= \frac{\sqrt{x}}{\log(x)}
\end{align*}

Now we apply the quotient rule for derivatives to show that the derivative of $g(x)$ is positive for $x > 8$.

\begin{align*}
g'(x) &= \frac{\log(x) \cdot \frac{1}{2}x^{-\frac{1}{2}} - \sqrt{x} \cdot \frac{1}{x}}{\log(x)^2} \\
&= \frac{\frac{1}{2}\log(x) - 1}{\sqrt{x}\log(x)^2}
\end{align*}

The latter is positive if, and only if, $\frac{1}{2}\log(x) - 1$ is positive, which happens for $x > e^2 \approx 7.389$.
\end{proof}

\begin{proof}[Proof of Lemma \ref{logbylinear}]
One can check that $\frac{2c_1\log(x)}{x}$ is smaller than $0.05$ for $x = 200$. Again applying the quotient rule, we see that its derivative is equal to $\frac{2c_1 - 2c_1\log(x)}{x^2}$. Since this is negative for $x > e$, we conclude that $\frac{2c_1\log(x)}{x}$ is a decreasing function for all $x > 200 > e$ and the claim follows.
\end{proof} 

\begin{proof}[Proof of Lemma \ref{basiclog}]
Writing $\log(x+a)$ as an integral and we get $\log(x+a) = \int_{z=1}^{x+a} \frac{1}{z} \, dz = \int_{z=1}^{x} \frac{1}{z} \, dz + \int_{z=x}^{x+a} \frac{1}{z} \, dz = \log(x) + \int_{z=x}^{x+a} \frac{1}{z} \, dz$. And the latter integral can be lower and upper bounded by $\frac{a}{x+a}$ and $\frac{a}{x}$, respectively.
\end{proof}

\begin{proof}[Proof of Lemmas \ref{taylorlog} and \ref{taylore}]
These immediately follow from looking at the appropriate Taylor series.

\begin{align*}
\log\left(1 + \frac{1}{x}\right) &= \sum_{k=1}^{\infty} \frac{(-1)^{k+1}}{kx^k} = \frac{1}{x} - \frac{1}{2x^2} + \frac{1}{3x^3} - \ldots > \frac{1}{x} - \frac{1}{2x^2} \\
e^{\frac{-1}{2x}} &= \sum_{k=0}^{\infty} \frac{(-1)^k}{(2x)^kk!} = 1 - \frac{1}{2x} + \frac{1}{2(2x)^2} - \ldots > 1 - \frac{1}{2x} \qedhere
\end{align*}
\end{proof}

\begin{proof}[Proof of Lemma \ref{nastylog}]
We will prove this Lemma for $n \ge 2^{18}$ by applying Lemmas \ref{taylorlog} and \ref{taylore} with $x = \frac{\log(\frac{n}{2})}{\log(2)}$. One can check the remaining values with a computer. 

\begin{align*}
\left(1 + \frac{1}{x}\right)^x &= e^{x \log\left(1 + \frac{1}{x}\right)} \\
&> e^{1 - \frac{1}{2x}} \\
&> e - \frac{e}{2x} \\
&= e - \frac{e \log(2)}{2 \log(\frac{n}{2})} \\
&> e - \frac{1}{\frac{18}{17} \log(\frac{n}{2})} \\
&= e - \frac{1}{\log(n) + (\frac{1}{17}\log(n) - \frac{18}{17}\log(2))} \\
&\ge e - \frac{1}{\log(n) + (\frac{1}{17}\log(2^{18}) - \frac{18}{17}\log(2))} \\
&= e - \frac{1}{\log(n)} \qedhere
\end{align*}
\end{proof}

\begin{proof}[Proof of Lemma \ref{upperbounde}]
We apply Lemma \ref{taylore} with $x = \frac{n+1}{2}$.

\begin{align*}
e &= \frac{1}{e^{-1}} \\
&= \frac{1}{\left(e^{\frac{-1}{n+1}} \right)^{n+1}} \\
&< \frac{1}{\left(1 - \frac{1}{n+1} \right)^{n+1}} \\
&= \left(1 + \frac{1}{n}\right)^{n+1} \qedhere
\end{align*}
\end{proof}

\begin{proof}[Proof of Lemma \ref{anothernastylog}]
By the change of variables $y = x+a$, the upper bound in Lemma \ref{basiclog} can be rewritten as $\log(y - a) > \log(y) - \frac{a}{y-a}$. Now we plug in $y = \log(n+1)$ and $a = \log(2)$ to get $\log(\log(n+1) - \log(2)) > \log(\log(n+1)) - \frac{\log(2)}{\log(n+1) - \log(2)}$. We will now apply the equalities $\log(n+1) - \log(2) = \log\left(\frac{n+1}{2}\right)$ and $2c_1\log(2) = 1$ to finish the proof;

\begin{align*}
2c_1\log\left(\log\left(\frac{n+1}{2}\right)\right) &= 2c_1\log(\log(n+1) - \log(2)) \\
&> 2c_1\log(\log(n+1)) - \frac{2c_1\log(2)}{\log(n+1) - \log(2)} \\
&= 2c_1\log(\log(n+1)) - \frac{1}{\log\left(\frac{n+1}{2}\right)} \\
&> 2c_1\log(\log(n)) - \frac{1}{\log\left(\frac{n+1}{2}\right)} \qedhere
\end{align*}
\end{proof}

\begin{proof}[Proof of Lemma \ref{basicloglog}]
Since the derivative of $\log(\log(x))$ is equal to $\frac{1}{x \log(x)}$ we get 

\begin{align*}
\log(\log(n+1)) - \log(\log(n)) &= \int_{x=n}^{n+1} \frac{1}{x \log(x)} \, dx \\
&< \frac{1}{n \log(n)} \qedhere
\end{align*}
\end{proof}

\begin{proof}[Proof of Lemma \ref{notherone}]
\begin{align*}
\frac{2c_1(n+2)\log(n+1)}{n\log(n)} &= 2c_1 \cdot \left(\frac{n+2}{n}\right) \cdot \left(\frac{\log(n+1)}{\log(n)}\right) \\
&< 2c_1\left(1 + \frac{2}{n}\right)\left(1 + \frac{1}{n \log(n)}\right) \\
&\le 1.5 \cdot 1.5 \cdot 1.2 \\
&< 3 \qedhere
\end{align*}
\end{proof}

\begin{proof}[Proof of Theorem \ref{growthweak}]
With all of this necessary evil out of the way, we can finally get to the proof of our lower and upper bounds, and we shall start with the lower bound. With a computer it can be checked that the inequality $a_n > n^{f(n)}$ holds for all $n$ with $141 \le n \le 10^5$. Now we use induction, so let $n \ge 10^5$ be a positive integer such that $a_k > k^{f(k)}$ for all $k$ with $141 \le k \le n$. In particular it holds for $k = n$ and $k = \left \lfloor \frac{n+1}{2} \right \rfloor$. Then we will prove $a_k > k^{f(k)}$ for $k = n+1$ as well. Note that we may assume $a_{\left \lfloor \frac{n+1}{2} \right \rfloor} > \left(\frac{n}{2}\right)^{f\left(\frac{n}{2}\right)}$ by the induction hypothesis and Lemma \ref{fngrows}.

\begin{align*}
a_{n+1} &= a_n + a_{\left \lfloor \frac{n+1}{2} \right \rfloor} \\
&> n^{f(n)} + \left(\frac{n}{2} \right)^{c_1 \log\left(\frac{n}{2}\right) - 2c_1\log(\log(\frac{n}{2})) + c_2} \\
&= n^{f(n)} + \left(\frac{n}{2} \right)^{f(n) - \frac{1}{2} + (2c_1\log(\log(n)) - 2c_1\log(\log(\frac{n}{2}))} \\
&= n^{f(n)} \left(1 + \frac{\sqrt{2}}{\sqrt{n} \cdot 2^{f(n)}} \cdot \left(\frac{n}{2}\right)^{2c_1\log(\log(n)) - 2c_1\log(\log(\frac{n}{2})} \right) \\
&= n^{f(n)} \left(1 +  \frac{2c_1\log(n)}{en} \cdot \left(\frac{n}{2}\right)^{2c_1\log(\log(n)) - 2c_1\log(\log(\frac{n}{2})} \right) \\
&= n^{f(n)} \left(1 +  \frac{2c_1\log(n)}{en} \cdot \left(1 + \frac{\log(2)}{\log(\frac{n}{2})}\right)^{\frac{\log(\frac{n}{2})}{\log(2)}} \right) \\
&> n^{f(n)} \left(1 + \frac{2c_1\log(n)}{en} \cdot \left(e - \frac{1}{\log(n)}\right)\right) \\
&= n^{f(n)} \left(1 + \frac{2c_1\log(n)}{n} - \frac{2c_1}{en} \right) \\
&> n^{f(n)} \left(1 + \frac{2c_1\log(n) - 0.05}{n-1} - \frac{2c_1}{e(n-1)} \right) \\
&> n^{f(n)} \left(1 + \frac{2c_1\log(n+1) - \frac{2c_1}{n} - 0.05 - \frac{2c_1}{e}}{n-1} \right) \\
&> n^{f(n)} \left(1 + \frac{2c_1\log(n+1) - 0.6}{n-1}\right) 
\end{align*}

Where the last inequality uses $n \ge 75$. On the other hand, we also have an upper bound on $(n+1)^{f(n+1)}$.

\begin{align*}
(n+1)^{f(n+1)} &< (n+1)^{f(n) + \frac{c_1}{n}} \\
&= n^{f(n)} \cdot \left(1 + \frac{1}{n}\right)^{f(n)} \cdot e^{\frac{c_1\log(n+1)}{n}} \\
&< n^{f(n)} \cdot \left(1 + \frac{1}{n-1}\right)^{f(n)} \cdot \left(\left(1+\frac{1}{n-1}\right)^n\right)^{\frac{c_1\log(n+1)}{n}} \\
&< n^{f(n)} \cdot \left(1 + \frac{1}{n-1}\right)^{2c_1\log(n+1) - 2c_1\log(\log(n)) + c_2}
\end{align*}

Here let us define $g(n) = \left\lceil2c_1\log(n+1) - 2c_1\log(\log(n)) + c_2\right\rceil$ in order to apply the binomial theorem. Note by the way that, for $n \ge 10^5$, we get $2c_1\log(\log(n)) > 3.52 > c_2 + 1 + 1.1$, implying $g(n) < 2c_1\log(n+1) - 1.1$.

\begin{align*}
n^{f(n)} \cdot \left(1 + \frac{1}{n-1}\right)^{2c_1\log(n+1) - 2c_1\log(\log(n)) + c_2} &\le n^{f(n)} \cdot \left(1 + \frac{1}{n-1}\right)^{g(n)} \\
&= n^{f(n)} \cdot \sum_{i = 0}^{g(n)} \binom{g(n)}{i} \frac{1}{(n-1)^i} \\
&= n^{f(n)} \cdot \left(1 + \frac{g(n)}{n-1} + \sum_{i = 2}^{g(n)} \binom{g(n)}{i} \frac{1}{(n-1)^i}\right) \\
&< n^{f(n)} \cdot \left(1 + \frac{g(n)}{n-1} + \frac{2^{g(n)}}{(n-1)^2} \right) \\
&< n^{f(n)} \cdot \left(1 + \frac{2c_1\log(n+1) - 1.1 + \frac{n+1}{2^{1.1}(n-1)}}{n-1}\right) \\
&< n^{f(n)} \cdot \left(1 + \frac{2c_1\log(n+1) - 0.6}{n-1}\right) \\
\end{align*}

Combining the bounds proves the desired lower bound on $a_{n+1}$. \\

To prove the upper bound, we also need a few definitions. First define $h(n) = \left(\frac{n+1}{2} \right)^{\frac{c_1}{n}} \cdot \left(1 + \frac{1}{n} \right)^{f(n) - \frac{1}{2} + \epsilon}$ and note that $h(n)$ converges to $1$ as $n$ goes to infinity. Now choose $N \ge 9$ large enough such that $2^{\epsilon} > \displaystyle \max_{n \ge N} \left(h(n)\textstyle\left(\frac{\log(n)}{\log(n)-\log(\log(n))-c_1^{-1}}\right)\left(\frac{n+1}{n}\right)\right)$. Such an $N$ exists since all terms in the product converge to $1$. And finally, define $C_{\epsilon} = a_N$. Then certainly we have $a_n \le a_N = C_{\epsilon} < C_{\epsilon}n^{f(n) + \epsilon}$ for all $n$ with $2 \le n \le N$. Now we once again use induction, so let $n \ge N$ be a positive integer such that $a_k < C_{\epsilon}k^{f(k) + \epsilon}$ for all $k$ with $2 \le k \le n$. Then we will prove $a_k < C_{\epsilon}k^{f(k) + \epsilon}$ for $k = n+1$ as well. \\

\begin{align*}
a_{n+1} &= a_n + a_{\left \lfloor \frac{n+1}{2} \right \rfloor} \\
&< C_{\epsilon}n^{f(n) + \epsilon} + C_{\epsilon}\left(\frac{n+1}{2} \right)^{c_1 \log\left(\frac{n+1}{2}\right) - 2c_1\log\left(\log\left(\frac{n+1}{2}\right)\right) + c_2 + \epsilon} \\
&< C_{\epsilon}n^{f(n) + \epsilon} + C_{\epsilon}\left(\frac{n+1}{2} \right)^{\left(c_1 \log\left(\frac{n}{2}\right)+\frac{c_1}{n}\right) - \left(2c_1\log(\log(n)) - \frac{1}{\log\left(\frac{n+1}{2}\right)}\right) +  c_2 + \epsilon} \\
&= C_{\epsilon}n^{f(n) + \epsilon} + C_{\epsilon}\left(\frac{n+1}{2} \right)^{\frac{1}{\log\left(\frac{n+1}{2}\right)}} \cdot\left(\frac{n+1}{2} \right)^{\frac{c_1}{n}} \cdot \left(1 + \frac{1}{n} \right)^{f(n) - \frac{1}{2} + \epsilon} \cdot \left(\frac{n}{2} \right)^{f(n) - \frac{1}{2} + \epsilon} \\
&= C_{\epsilon}n^{f(n) + \epsilon} + C_{\epsilon}eh(n)\left(\frac{n}{2} \right)^{f(n) - \frac{1}{2} + \epsilon} \\
&= C_{\epsilon}n^{f(n) + \epsilon} \left(1 + \frac{eh(n)}{2^{c_2 - \frac{1}{2} + \epsilon}} \cdot \frac{\log(n)}{n} \right) \\
&< C_{\epsilon}n^{f(n) + \epsilon} \left(1 + \frac{2c_1\log(n) - 2c_1\log(\log(n)) - 2}{n+1} \right) 
\end{align*}

On the other hand, we also have a lower bound on $C_{\epsilon}(n+1)^{f(n+1) + \epsilon}$.

\begin{align*}
C_{\epsilon}(n+1)^{f(n+1) + \epsilon} &= C_{\epsilon}n^{f(n) + \epsilon}\left(1 + \frac{1}{n}\right)^{f(n) + \epsilon} \left(n+1\right)^{f(n+1) - f(n)} \\
&> C_{\epsilon}n^{f(n) + \epsilon} \left(1 + \frac{1}{n+1}\right)^{f(n) + \epsilon} \left(n+1\right)^{\frac{c_1}{n+1}} \left(n+1\right)^{\frac{-2c_1}{n \log(n)}}\\
&= C_{\epsilon}n^{f(n) + \epsilon} \left(1 + \frac{1}{n+1}\right)^{f(n) + \epsilon} e^{\frac{c_1\log(n+1)}{n+1}} e^{\frac{-2c_1\log(n+1)}{n \log(n)}} \\
&> C_{\epsilon}n^{f(n) + \epsilon} \left(1 + \frac{1}{n+1}\right)^{f(n) + \epsilon + (n+1)\left(\frac{c_1\log(n+1)}{n+1}\right) + (n+2)\left(\frac{-2c_1\log(n+1)}{n\log(n)}\right)}  \\
&> C_{\epsilon}n^{f(n) + \epsilon} \left(1 + \frac{1}{n+1}\right)^{f(n) + \epsilon + c_1 \log(n) - 3} \\
&> C_{\epsilon}n^{f(n) + \epsilon} \left(1 + \frac{1}{n+1}\right)^{2c_1 \log(n) - 2c_1\log(\log(n)) - 2} \\
&> C_{\epsilon}n^{f(n) + \epsilon} \left(1 + \frac{2c_1\log(n) - 2c_1\log(\log(n)) - 2}{n+1}\right)  \\
\end{align*}

We again obtain the desired inequality by combining the upper and lower bounds.

\end{proof}

\section{Concluding thoughts and remarks}
Lemmas \ref{fourtimesodd}, \ref{fourtimeseven}, \ref{oddeight} and \ref{eveneight} can all be suitably generalized. For example, the congruence $a_{4n} \equiv a_{n} \pmod{8}$ from Lemma \ref{fourtimeseven} actually holds modulo $16$ and can be even be further generalized to $a_{4n} \equiv a_{n} + 8n \pmod{32}$. The only thing that needs to be added to the proof is that, at the final equality, one should realize that for even $n$, $a_{2n+2}$ is divisible by $2$ but not by $4$. The congruence from Lemma \ref{fourtimesodd} can similarly be be generalized to an equality that holds modulo $32$, by rewriting it as $a_{4n} \equiv 5a_n \pmod{32}$. As for the statement of Lemma \ref{oddeight}, it can be slightly strengthened by noting that there are only two distinct possibilities; either $(a_{8n+1}, a_{8n+3}, a_{8n+5}, a_{8n+7}) \pmod{8} = (1, 3, 7, 5)$ or $(a_{8n+1}, a_{8n+3}, a_{8n+5}, a_{8n+7}) \pmod{8} = (7, 5, 1, 3)$. This can be shown by conditioning on the parity of $a_{4n+4}$. If $a_{4n+4}$ is even, then $a_{8n+9} = a_{8n+7} + 2a_{4n+4} \equiv a_{8n+7} + 4 \pmod{8}$. Whereas if $a_{4n+4}$ is odd, then $a_{4n+4} = a_{4n+3} + a_{2n+2} \equiv a_{4n+3} + 2 \pmod{4}$, so that $a_{8n+9} = a_{8n+5} + 2a_{4n+4} + 2a_{4n+3} \equiv a_{8n+5} \pmod{8}$. In both cases we are done by induction by applying the $4$ congruences from the proof of Lemma \ref{oddeight}. Finally, the proof of Lemma \ref{eveneight} essentially shows that the set $\{a_{16n+2}, a_{16n+6}, a_{16n+10}, a_{16n+14} \}$ contains all residue classes modulo $16$ that are divisible by $2$ but not by $4$. \\

The above remarks can potentially be used to resolve Conjecture \ref{conj} for $m = 16$. And of course, with more efficient coding or more computing power so that more values of $j$ can be checked, all lower bounds from Theorem \ref{minimumes} can be improved without too much effort. On the other hand, for any odd value of $m$, we have not yet even proved that the upper density of $n$ such that $a_n$ is divisible by $m$ is bounded away from $1$. As we mentioned before, Theorem \ref{densitytheory} is unfortunately not strong enough for this purpose, so it would require a new idea. That being said, we have proven, for example, that every residue class modulo $5$ occurs infinitely often, and so do $1 \pmod{6}$ and $5 \pmod{6}$. Since these have not yet led to a positive lower density however, we have decided to omit these proofs. Finally, it would be nice if the existence of $d_{x, 32}$ could be resolved for odd $x$. \\

None of these problems seem out of reach necessarily, so anyone is encouraged to grab a pen and piece of paper, and start thinking. 

\section{Acknowledgements}
The author would like to thank Mar Curc\'o-Iranzo for helpful comments, even though they were mostly ignored.

\end{document}